\documentclass[twoside, 12pt]{amsart}
\usepackage{latexsym}
\usepackage{amssymb,amsmath,amsopn}
\usepackage[dvips]{graphicx}   
\usepackage{color,epsfig}      

\newtheorem{theorem}{Theorem}[]
\newtheorem{corollary}[]{Corollary}[]
\newtheorem{lemma}[]{Lemma}[]

\newtheorem{problem}[]{Problem}[]

\theoremstyle{definition}
\newtheorem{defi}[]{Definition}[]
\newtheorem{remark}[]{Remark}[]

\numberwithin{equation}{section}

\DeclareMathOperator{\conv}{conv}
\DeclareMathOperator{\Id}{Id}
\DeclareMathOperator{\Sym}{Sym}

\DeclareMathOperator{\vol}{vol}

\DeclareMathOperator{\inter}{int}
\DeclareMathOperator{\relint}{relint}

\DeclareMathOperator{\bd}{bd}

\DeclareMathOperator{\aff}{aff}

\newcommand{\R}{\mathbb{R}}
\newcommand{\Z}{\mathbb{Z}}

\newcommand{\B}{\mathbf B}
\newcommand{\Sph}{\mathbb{S}}
\renewcommand{\P}{\mathcal{P}}
\newcommand{\F}{\mathcal{F}}
\newcommand{\C}{\mathcal{C}}

\begin{document}

\title[Maximum volume polytopes]{Maximum volume polytopes inscribed in the unit sphere}
\author[\'A. G.Horv\'ath and Z. L\'angi]{\'Akos G.Horv\'ath and Zsolt L\'angi}

\address{\'Akos G.Horv\'ath, Dept. of Geometry, Budapest University of Technology,
Egry J\'ozsef u. 1., Budapest, Hungary, 1111}
\email{ghorvath@math.bme.hu}
\address{Zsolt L\'angi, Dept. of Geometry, Budapest University of Technology,
Egry J\'ozsef u. 1., Budapest, Hungary, 1111}
\email{zlangi@math.bme.hu}


\begin{abstract}
In this paper we investigate the problem of finding the maximum volume polytopes, inscribed in the unit sphere of the $d$-dimensional Euclidean space, with a given number of vertices. We solve this problem for polytopes with $d+2$ vertices in every dimension, and for polytopes with $d+3$ vertices in odd dimensions. For polytopes with $d+3$ vertices in even dimensions we give a partial solution.
\end{abstract}

\maketitle

\section{Introduction}
To find the convex polyhedra in Euclidean $3$-space $\R^3$, with a given number of faces and with minimal isoperimetric quotient, is a centuries old question
of geometry: research in this direction perhaps started with the work of Lhuilier in the 18th century. A famous result of Lindel\"of \cite{L70}, published in the 19th century, yields a necessary condition for such a polyhedron: it states that any optimal polyhedron is circumscribed about a Euclidean ball, and this ball touches each face at its centroid. In particular, it follows from his result that, instead of fixing surface area while looking for minimal volume, we may fix the inradius of the polyhedron. Since the publication of this result, the same condition for polytopes in $d$-dimensional space $\R^d$ has been established (cf. \cite{G07}), and many variants of this problem have been investigated (cf., e.g. \cite{BB96}).
For references and open problems of this kind, the interested reader is referred to \cite{F93}, \cite{croft} or \cite{BMP05}.

We mention just one of them in more detail: the aim of the discrete isodiametric problem is to find, among $d$-polytopes with a given number of vertices, the ones
with maximal volume. For polytopes with $(d+2)$ vertices this question was answered by Kind and Kleinschmidt \cite{KK76}.
The solution for polytopes with $d+3$ vertices was published in \cite{KW03}, which later turned out to be incomplete (cf. \cite{KW05}), and thus,
this case is still open.

The dual of the original problem: to find the maximal volume polyhedra in $\R^3$ with a given number of vertices and inscribed in the unit sphere,
was first mentioned in \cite{ftl} in 1964.
A systematic investigation of this question was started with the paper \cite{bermanhanes} of Berman and Hanes in 1970, who found a necessary condition
for optimal polyhedra, and determined those with $n \leq 8$ vertices.
The same problem was examined in \cite{M02}, where the author presented the results of a computer-aided search for optimal polyhedra with $4 \leq n \leq 30$ vertices.
Nevertheless, according to our knowledge, this question, which is listed in both research problem books \cite{BMP05} and \cite{croft}, is still open for polyhedra with $n > 8$ vertices.

The aim of our paper is to investigate this problem for polytopes in arbitrary dimensions.
In Section~\ref{sec:prelim}, we collect the tools necessary for the investigation, and, in particular, by generalizing the methods of \cite{bermanhanes},
present a necessary condition for the optimality of a polytope.
In Section~\ref{sec:dplus2}, we find the maximum volume polytopes in $\R^d$, inscribed in the unit sphere $\Sph^{d-1}$, with $n=d+2$ vertices.
For $n=d+3$ vertices, our results can be found in Section~\ref{sec:dplus3}. In particular, we find the maximum for $d$ odd, over the family of all polytopes,
and the local maximizers for $d$ even, in the family of not cyclic polytopes.
Finally, in Section~\ref{sec:remarks}, we collect our remarks and propose some open problems.

In the paper, for any $p,q \in \R^d$, $|p|$ and $[p,q]$ denote the standard Euclidean norm of $p$, and the closed segment with endpoints $p$ and $q$,respectively.
The origin of the standard coordinate system of $\R^d$ is denoted by $o$.
If $v_1,v_2,\ldots,v_d \in \R^d$, then the $d \times d$ determinant with columns $v_1,v_2,\ldots,v_d$,
in this order, is denoted by $|v_1,\ldots,v_d|$.
The unit ball of $\R^d$, with $o$ as its center, is denoted by $\B^d$, and we set $\Sph^{d-1}=\bd \B^d$.

Throughout the paper, by a polytope we mean a convex polytope.
The vertex set of a polytope $P$ is denoted by $V(P)$.
We denote the family of $d$-dimensional polytopes, with $n$ vertices and inscribed in the unit sphere $\Sph^{d-1}$, by
$\P_d(n)$. We denote $d$-dimensional volume by $\vol_d$, and set $v_d(n) = \max \{ \vol_d(P) : P \in \P_d(n) \}$.
Note that by compactness, $v_d(n)$ exists for any value of $d$ and $n$.

\section{Preliminaries}\label{sec:prelim}

Let $P$ be a $d$-polytope inscribed in the unit sphere $\Sph^{d-1}$, and let $V(P) = \{ p_1,p_2,\ldots,p_n\}$.

Let $\C(P)$ be a simplicial complex with the property that $|\C(P)| = \bd P$, and that the vertices of $\C(P)$ are exactly the points of $V(P)$.
We orient $\C(P)$ in such a way that for each $(d-1)$-simplex $(p_{i_1},p_{i_2},\ldots, p_{i_d})$ in $C(P)$,
the determinant $|p_{i_1},\ldots,p_{i_d}|$ is positive; and call the $d$-simplex $\conv \{ o,p_{i_1},\ldots,p_{i_d}\}$ a \emph{facial simplex} of $P$.
We call the $(d-1)$-dimensional simplices of $\C(P)$ the \emph{facets} of $\C(P)$.

For the next definition, for $3$-polytopes, see also Section 2 of \cite{bermanhanes}.

\begin{defi}
Let $P \in \P_d(n)$ be a $d$-polytope with $V(P) = \{p_1,p_2,\ldots ,p_n\}$.
If for each $i$, there is an open set $U_i\subset \Sph^{d-1}$ such that $p_i\in U_i$, and for any $q \in U_i$, we have
\[
\vol_d\left( \left( \conv \left( V(P) \setminus \{p_i \} \right) \right) \cup \{ q \} \right) \leq \vol_d\left(P\right),
\]
then we say that $P$ satisfies \emph{Property Z}.
\end{defi}

\begin{lemma}\label{lem:propz}
Consider a polytope $P \in \P_d(n)$ satisfying Property Z.
For any $p \in V(P)$, let $\F_p$ denote the family of the facets of $\C(P)$ containing $p$.
For any $F \in \F_p$, set $A(F,p) = \vol_{d-1} \left( \conv \left( \left( V(F) \cup \{ o \} \right) \setminus \{ p \} \right) \right)$, and let $m(F,p)$ be the unit normal vector of the hyperplane, spanned by $\left( V(F) \cup \{ o \} \right) \setminus \{p\}$, pointing in the direction of the half space containing $p$.
Then
\begin{enumerate}
\item[(\ref{lem:propz}.1)] we have
\[
p = \frac{m}{|m|}, \mbox{ where } m=\sum_{F \in \F_p} A(F,p) m(F,p), \mbox{ and}
\]
\item[(\ref{lem:propz}.2)] $P$ is simplicial.
\end{enumerate}
\end{lemma}

\begin{proof}
For any $p \in V(P)$, the volume of $P$ can be written as
\[
\vol_d(P)= V + \sum_{F \in \F_p} \vol_d \left( \conv (F \cup \{ o \} )\right) = V + \sum_{F \in \F_p} \frac{1}{d} A(F,p) \langle m(F,p),p \rangle ,
\]
where $V$ is the sum of the volumes of the facial simplices of $\C(P)$, \emph{not} containing $p$.
From this it follows that
\[
\vol_d(P) = V + \frac{1}{d} \langle p, m \rangle.
\]
Now, since $p \in \Sph^{d-1}$, if $\langle p,m \rangle < |m|$, then for any open set $U \subset \Sph^{d-1}$, there is a point $q \in U$
such that $\langle p,m \rangle < \langle q,m \rangle$, implying that
$\vol_d(P) < \vol_d\left(\conv \left( \left( V(P) \cup \{ q \} \right) \setminus \{p\} ) \right) \right)$, which contradicts our assumption
that $P$ satisfies Property Z. Thus, we have $\langle p, m \rangle=|m|$, or in other words, $p=\frac{m}{|m|}$.

Now assume that $P$ is not simplicial, that is, that some facet $F$ of $P$ is a simplex.
First, we consider the case that $F$ has $d+1$ vertices, and, as a $(d-1)$-polytope, it is simplicial.
Then $F$ can be written in the form $F=\conv (S_1 \cup S_2)$, where $S_1$ and $S_2$ are two simplices with $\dim S_1 + \dim S_2 = \dim F = d-1$,
and $S_1 \cap S_2$ is a singleton $\{ r \}$ in the relative interiors of both $S_1$ and $S_2$ (cf. \cite{grunbaum}).
Let $V(S_1) = \{ p_i : i=1,2,\ldots, m \}$, and $V(S_2)= \{ q_j : j=1,2,\ldots, d+1-m \}$.
Then we may triangulate $F$ in two different ways:
\[
\F_1 = \left\{ \conv (V(F) \setminus \{ p_i \}): i=1,2,\ldots,m \right\}
\]
and
\[
\F_2 = \left\{ \conv (V(F) \setminus \{ q_i \}): i=1,2,\ldots,d+1-m \right\}.
\]
Observe that the union of the elements of $\F_1$ containing $p_1$ is the closure of $F \setminus \conv (V(F) \setminus \{ p_1 \})$, whereas for $\F_2$ it is $F$.
Recall that for any simplex in $\R^s$, with external unit facet normals $m_1, m_2, \ldots, m_{s+1}$ belonging to the facets $F_1, F_2, \ldots, F_{s+1}$, respectively,
we have
\begin{equation}\label{eq:notsimplicial}
\sum_{i=1}^{s+1} \vol_{s-1}(F_i) m_i = 0.
\end{equation}
On the other hand, since the quantity in (\ref{lem:propz}.1) must be independent from the triangulation, by (\ref{eq:notsimplicial}) we have reached a contradiction.
We remark that if $F$ is not simplicial, then $F$ is a $(d-k-1)$-fold pyramid over a $k$-polytope with $(k+2)$ vertices (cf. \cite{grunbaum}), for which a straightforward modification of our argument yields the statement.

Finally, we consider the case that $F$ has more than $d+1$ vertices.
Choose a set $S$ of $d+1$ vertices of $F$.
Note that any triangulation of $\conv S$ for any $S \subseteq V(F)$ can be extended to a triangulation of $F$; this can be easily shown by induction.
Thus, the assertion follows by applying the argument of the previous paragraph for $\conv S$.
\end{proof}

\begin{remark}\label{rem:hyperplane}
Assume that $P \in \P_d(n)$ satisfies Property Z, and for some $p \in V(P)$, all the vertices of $P$ adjacent to $p$ are contained in a hyperplane $H$.
Then the supporting hyperplane of $\Sph^{d-1}$ at $p$ is parallel to $H$, or in other words, $p$ is a normal vector to $H$.
Thus, in this case all the edges of $P$, starting at $p$, are of equal length.
\end{remark}

\begin{lemma}\label{lem:note2}
Let $P \in \P_d(n)$ satisfy Property Z, and let $p \in V(P)$. Let $q_1, q_2 \in V(P)$ be adjacent to $p$.
Assume that any facet of $P$ containing $p$ contains at least one of $q_1$ and $q_2$, and for any $S \subset V(P)$ of cardinality $d-2$, $\conv (S \cup \{p,q_1\})$ is a facet of $P$ not containing $q_2$ if, and only if $\conv (S \cup \{p,q_2\})$ is a facet of $P$ not containing $q_1$. Then $|q_1-p| = |q_2-p|$.
\end{lemma}

\begin{proof}
Let
\[
U = \{ S \subset V(P): \conv(S \cup\{p,q_1\} ) \mbox{ is a facet of } P, \mbox{ but } p,q_1,q_2 \notin S\},
\]
and let
\[
W = \{ S \subset V(P): \conv(S \cup\{p,q_1,q_2\} ) \mbox{ is a facet of } P, \mbox{ but } p,q_1,q_2 \notin S\}.
\]

Note that by our conditions, $U$ is also the family of subsets of $V(P)$, not containing $p,q_1,q_2$, such that $\conv (S \cup \{p,q_2 \} )$
is a facet of $P$.

Then, using a suitable labelling of the vertices of $P$, the total volume $V$ of the facial simplices containing $p$ can be written as
\[
V = \frac{1}{d!} \sum_{ \{ r_{i_1},\ldots,r_{i_{d-2}} \} \in U} \left( |p,q_1,r_{i_1},\ldots,r_{i_{d-2}}| - |p,q_2,r_{i_1},\ldots,r_{i_{d-2}}| \right) +
\]
\[
+\frac{1}{d!} \sum_{ \{ r_{i_1},\ldots,r_{i_{d-3}} \} \in W} |p,q_1,q_2,r_{i_1},\ldots,r_{i_{d-3}}|.
\]

Observe that for any $\{ r_{i_1},\ldots,r_{i_{d-3}} \} \in W$, we have $|p,q_2,q_2,r_{i_1},\ldots,r_{i_{d-3}}| = 0$, which, for every element of $W$, we may subtract from $V$ without changing its value.
Thus, for some suitable finite set $X \subset (\R^d)^{d-2}$ of $(d-2)$-tuples of points in $\R^d$, we have
\[
V = \frac{1}{d!} \sum_{(v_1,\ldots,v_{d-2}) \in X} |p,q_1-q_2,v_1,\ldots,v_{d-2}|.
\]

Let $f:\R^d \to \R$ be the linear functional $f(x) = \frac{1}{d!} \sum_{(v_1,\ldots,v_{d-2}) \in X} |x,q_1-q_2,v_1,\ldots,v_{d-2}|$.
Since $P$ satisfies Property Z, we have that $p$ is a normal vector of the hyperplane $\{ x \in \R^d : f(x) = 0\}$.
On the other hand, $f(q_2-q_1) = 0$, due to the properties of determinants.
Thus, $q_2-q_1$ and $p$ are perpendicular, from which it readily follows that $\langle p,q_1 \rangle = \langle p,q_2 \rangle$, and hence, $|q_1-p| = |q_2-p|$.
\end{proof}

Corollary~\ref{cor:simplex} is a straightforward consequence of Lemma~\ref{lem:note2} or, equivalently, Remark~\ref{rem:hyperplane}.

\begin{corollary}\label{cor:simplex}
If $P \in \P_d(d+1)$ and $\vol_d(P) = v_d(d+1)$, then $P$ is a regular simplex inscribed in $\Sph^{d-1}$.
\end{corollary}

For Corollary~\ref{cor:bipyramid}, recall that if $K$ is a $(d-1)$-polytope in $\R^d$, and $[p_1,p_2]$ is a segment intersecting the relative interior of $K$
at a singleton different from $p_1$ and $p_2$, then $\conv (K \cup [p_1,p_2])$ is a \emph{$d$-bipyramid} with base $K$ and apexes $p_1,p_2$ (cf. \cite{grunbaum}).

\begin{corollary}\label{cor:bipyramid}
Let $P \in \P_d(n)$ be combinatorially equivalent to a $d$-bipyramid. Assume that $P$ satisfies Property Z.
Then $P$ is a $d$-bipyramid, its apexes $p_1,p_2$ are antipodal points, its base $K$ and $[p_1,p_2]$ lie in orthogonal linear subspaces of $\R^d$, and $K$ satisfies Property Z in the hyperplane $\aff K$.
\end{corollary}

\begin{proof}
Let $p_1$ and $p_2$ be the vertices of $P$ corresponding to the apexes of the combinatorially equivalent $d$-bipyramid.
Then, by Lemma~\ref{lem:note2}, for any $q \in V(P) \setminus \{p_1,p_2\}$, we have $|p_2-q| = |p_1-q|$.
Thus, $P$ is a $d$-bipyramid with apexes $p_1,p_2$ and base $K$.
Furthermore, $K$ and $[p_2,p_1]$ are in orthogonal linear subspaces of $\R^d$.
Finally, applying Lemma~\ref{lem:propz} to $K$, we obtain that $K$ satisfies Property Z in $\aff K$.
\end{proof}

Corollary~\ref{cor:crosspolytope} is a special case of Corollary~\ref{cor:bipyramid}.

\begin{corollary}\label{cor:crosspolytope}
If $P \in \P_{d}(2d)$ has maximal volume in the combinatorial class of cross-polytopes inscribed in $\Sph^{d-1}$, then it is a regular cross-polytope.
\end{corollary}

\section{Polytopes with $n=d+2$ vertices}\label{sec:dplus2}

In the proof of our results on $d$-polytopes with $d+2$ or $d+3$ vertices, we use extensively the properties of the so-called Gale transform of a polytope
(cf. \cite{grunbaum}, \cite{ziegler}).

Consider a $d$-polytope $P$ with vertex set $V(P)=\{p_i:i=1,2,\ldots,n \}$. Regarding $\R^d$ as the hyperplane $\{ x_{d+1} = 1\}$ of $\R^{d+1}$, we can represent $V(P)$ as a $(d+1) \times n$ matrix $M$, in which each column lists the coordinates of a corresponding vertex in the standard basis of $\R^{d+1}$. Clearly, this matrix has rank $d+1$, and thus, it defines a linear mapping $L : \R^n \rightarrow \R^{d+1}$, with $\dim \ker L= n-d-1$.
Consider a basis $\{ w_1,w_2,\ldots, w_{n-d-1} \}$ of $\ker L$, and let $\bar{L}:\R^{n-d-1} \rightarrow \R^n$ be the linear map mapping the $i$th vector of the standard basis of $\R^{n-d-1}$ into $w_i$.
Then the matrix $\bar{M}$ of $\bar{L}$ is an $n\times (n-d-1)$ matrix of (maximal) rank $n-d-1$, satisfying the equation $M \bar{M} = O$, where $O$ is the matrix with all entries equal to zero. Note that the rows of $\bar{M}$ can be represented as points of $\R^{n-d-1}$.
For any vertex $p_i \in V(P)$, we call the $i$th row of $\bar{M}$ the \emph{Gale transform of $p_i$}, and denote it by $\bar{p}_i$.
Furthermore, the $n$-element multiset $\{ \bar{p_i} : i=1,2,\ldots, n \} \subset \R^{n-d-1}$ is called the \emph{Gale transform of $P$}, and is denoted by $\bar{P}$. If $\conv S$ is a face of $P$ for some $S \subset V(P)$, then the (multi)set of the Gale transform of the points of $S$ is called  a face of $\bar{P}$.
If $\bar{S}$ is a face of $\bar{P}$, then $\bar{P} \setminus \bar{S}$ is called a \emph{coface} of $\bar{P}$.

Let $V = \{ q_i: i=1,2,\ldots,n\} \subset \R^{n-d-1}$ be a (multi)set. We say that $V$ is a \emph{Gale diagram} of $P$, if for some Gale transform $P'$
the conditions $o \in \relint \conv \{ q_j : j \in I \}$ and $o \in \relint \conv \{ \bar{p}_j : j \in I \}$ are satisfied for the same subsets of $\{1,2,\ldots,n\}$. If $V \subset \Sph^{n-d-2}$, then $V$ is a \emph{normalized Gale diagram} (cf. \cite{lee}). A \emph{standard Gale diagram} is a normalized Gale diagram in which the consecutive diameters are equidistant. A \emph{contracted Gale diagram} is a standard Gale diagram which has the least possible number of diameters among all isomorphic diagrams. We note that each $d$-polytope with at most $d+3$ vertices may be represented by a contracted Gale diagram (cf. \cite{grunbaum} or \cite{ziegler}).

In the proofs, we need the following theorem from \cite{grunbaum} or also from \cite{ziegler}.

\begin{theorem}[\cite{grunbaum},\cite{ziegler}]\label{thm:Gale}
\begin{itemize}
\item[(i)] A multiset $\bar{P}$ of $n$ points in $\R^{n-d-1}$ is a Gale diagram of a $d$-polytope
$P$ with $n$ vertices if and only if every open half-space in $\R^{n-d-1}$ bounded by a hyperplane
through $o$ contains at least two points of $\bar{V}$ (or, alternatively, all the points of $\bar{P}$
coincide with $o$ and then $n=d+1$ and $P$ is a  $d$-simplex).
\item[(ii)] If $F$ is a facet of $P$, and $Z$ is the corresponding coface, then in any Gale
diagram $\bar{V}$ of $P$, $\bar{Z}$ is the set of vertices of a (non-degenerate) set with $o$ in its
relative interior.
\item[(iii)] A polytope $P$ is simplicial if and only if, for every hyperplane $H$ containing
$o  \in \R^{n-d-1}$, we have $o \notin \relint \conv (\bar{V} \cap H)$.
\item[(iv)] A polytope $P$ is a pyramid if and only if at least one point of $\bar{V}$ coincides with
the origin $o \in \R^{n-d-1}$.
\end{itemize}
\end{theorem}

\begin{remark}\label{rem:Galediagram}
We note that (ii) can be stated in a more general form: $F$ is a face of $P$ if, and only if, for the corresponding co-face $\bar{F}$ of $P$, we have $o \in \inter \conv \bar{Z}$.
\end{remark}

The main result in this section is as follows.

\begin{theorem}\label{thm:dplus2}
Let $P \in \P_d(d+2)$ have maximal volume over $\P_d(d+2)$.
Then $P=\conv (P_1 \cup P_2)$, where $P_1$ and $P_2$ are regular simplices of dimensions $\lfloor \frac{d}{2} \rfloor$ and $\lceil \frac{d}{2} \rceil$, respectively, inscribed in $\Sph^{d-1}$, and contained in orthogonal linear subspaces of $\R^d$.
Furthermore,
\[
v_d(d+2) = \frac{1}{d!} \cdot \frac{\left( \lfloor d/2 \rfloor + 1 \right)^{\frac{\lfloor d/2 \rfloor + 1}{2}} \cdot \left( \lceil d/2 \rceil + 1 \right)^{\frac{\lceil d/2 \rceil + 1}{2}}}{\lfloor d/2 \rfloor^{\frac{\lfloor d/2 \rfloor}{2}} \cdot \lceil d/2 \rceil^{\frac{\lceil d/2 \rceil}{2} }}
\]
\end{theorem}

\begin{proof}
Without loss of generality, we may assume that $d \geq 3$, as otherwise the assertion is trivial.

In the proof we use a contracted Gale diagram $\bar{P}$ of $P$.
Since by Lemma~\ref{lem:propz} $P$ is simplicial, and since $d+2-d-1=1$, (iii) of Theorem~\ref{thm:Gale} yields that $\bar{P}$
consists of the points $-1$ and $1$ on the real line. We may assume that the multiplicity of $-1$ is $k+1$ and that of $1$ is $d+1-k$.
From (i) of Theorem~\ref{thm:Gale}, it follows that $2 \leq k \leq d$.
Without loss of generality, we may assume that $k+1 \leq d+1-k$, or in other words, that $k \leq \lfloor \frac{d}{2} \rfloor$.
By (ii) of Theorem~\ref{thm:Gale}, the facets of $\bar{P}$ are the complements of the pairs of the form $\{-1,1\}$. 

Let $V_+(P)$ be the set of vertices of $P$ represented by $1$ in $\bar{P}$, and let $V_-(P) = V(P) \setminus V_+(P)$.
Consider any $p \in V_+(P)$ and $q_1,q_2 \in V_-(P)$. Observe that both $q_1$ and $q_2$ are adjacent to $p$.
Furthermore, for $P$ and these three vertices the conditions of Lemma~\ref{lem:note2} are satisfied, which yields that $|q_2-p| = |q_1-p|$.
Hence, there is some $\delta > 0$ such that for any $p \in V_+(P)$ and $q \in V_-(P)$, we have $|q-p| = \delta$.
Thus, $V_+(P)$ and $V_-(P)$ are contained in orthogonal linear subspaces.
Since $P$ is $d$-dimensional, it follows that these subspaces are orthogonal complements of each other.

Let $P_1 = \conv V_+(P)$ and $P_2 = \conv V_-(P)$. Then $P_1$ is a $k$-dimensional, and $P_2$ is a $(d-k)$-dimensional simplex, and we have
\[
\vol_d(P) = \frac{k! (d-k)!}{d!} \vol_k(P_1) \vol_{d-k}(P_2).
\]
Now we can apply Corollary~\ref{cor:simplex}, which yields that $P_1$ and $P_2$ are regular.

It is well-known (and can be easily computed from its standard representation in $\R^k$) that the volume of a regular $k$-dimensional simplex
inscribed in $\Sph^{k-1}$ is $\frac{\left( k+1 \right)^{\frac{k+1}{2}}}{k^{\frac{k}{2}} k!}$.
Hence, we have
\[
\vol_d(P) = \frac{k! (d-k)!}{d!} \frac{\left( k+1 \right)^{\frac{k+1}{2}}}{k^{\frac{k}{2}} k!} \frac{\left( d-k+1 \right)^{\frac{d-k+1}{2}}}{(d-k)^{\frac{d-k}{2}} (d-k)!},
\]
or equivalently,
\[
\vol_d(P) = \frac{1}{d!}  \sqrt{ \left( 1+\frac{1}{k} \right)^k \left( 1+\frac{1}{d-k} \right)^{d-k}} \sqrt{(k+1)(d-k+1)}.
\]
We need to maximize this quantity for $k=1,2,\ldots,\lfloor \frac{d}{2} \rfloor$.
If $d$ is even, the assertion follows from the inequality for the arithmetic and the geometric means. If $d$ is odd, we may use the strict concavity of the
function $x \mapsto x \log \left( 1+\frac{1}{x} \right)$, $x > 0$.
\end{proof}

\begin{remark}
The use of Gale diagrams in the proof of Theorem~\ref{thm:dplus2} can be avoided if we recall the fact that any simplicial $d$-polytope is the convex hull of two simplices, having a single point, contained in the relative interiors of both simplices, as their intersection (cf. \cite{ziegler}).
The vertex sets of these simplices form the unique Radon partition of the point set.
\end{remark}

\section{Polytopes with $n=d+3$ vertices}\label{sec:dplus3}

Our main result is the following.
Before stating it, recall that a $d$-polytope with $n$ vertices is \emph{cyclic}, if it is combinatorially equivalent to the convex hull of $n$ points
on the moment curve $\gamma(t) = (t,t^2,\ldots, t^d)$, $t \in \R$.

\begin{theorem}\label{thm:dplus3}
Let $P \in \P_d(d+3)$ satisfy Property Z. If $P$ is even, assume that $P$ is not cyclic. Then
$P=\conv \{ P_1 \cup P_2 \cup P_3\}$, where $P_1$, $P_2$ and $P_3$ are regular simplices inscribed in $\Sph^{d-1}$ and contained in three mutually orthogonal linear subspaces of $\R^d$.
Furthermore:
\begin{itemize}
\item If $d$ is odd and $P$ has maximal volume over $\P_d(d+3)$, then the dimensions of $P_1$, $P_2$ and $P_3$ are $\lfloor d/3 \rfloor$ or 
$\lceil d/3 \rceil$. In particular, in this case we have
\[
\left( v_d(d+3) = \right) \vol_d(P) = \frac{1}{d!} \cdot \prod_{i=1}^3 (k_i+1)^{\frac{k_i+1}{2}} k_i^{- \frac{k_i}{2}},
\]
where $k_1+k_2+k_3 = d$ and for every $i$, we have $k_i \in \left\{ \lfloor \frac{d}{3} \rfloor,  \lceil \frac{d}{3} \rceil \right\}$. 
\item The same holds for the dimensions of $P_1$, $P_2$ and $P_3$, if $d$ is even, and $P$ has maximal volume over the family of \emph{not cyclic} elements of $\P_d(d+3)$, satisfying Property Z.
\end{itemize}
\end{theorem}

\begin{proof}
To prove the assertion, we use a contracted Gale diagram $\bar{P}$ of $P$. Since by Lemma~\ref{lem:propz} $P$ is simplicial, $\bar{P}$
is a multiset consisting of the vertices of a regular $(2k+1)$-gon, with $k \geq 1$ and the origin $o \in \R^2$ as its center,
such that the multiplicity of each vertex is at least one, and the sum of their multiplicities is $d+3$.

Applying Remark~\ref{rem:Galediagram}, we have that $p,q \in V(P)$ are not adjacent if, and only if, there is an open half plane,
containing $o$ in its boundary, that contains only the two points $\bar{p}, \bar{q}$ of $\bar{P}$.
In this case we have one of the following:
\begin{itemize}
\item $k=2$, and the points are two consecutive vertices of the pentagon $\bar{P}$, with multiplicity one.
\item $k=1$, and the points are either consecutive vertices of $G(P)$ with multiplicity one, or belong to the same vertex of $G(P)$, which has multiplicity exactly two.
\end{itemize}

Now, consider the case that some point of $\bar{P}$ has multiplicity greater than one, and let $p_1,p_2,\ldots,p_m \in V(P)$ be represented by this point. We set $V_1 = \{ p_1,p_2,\ldots,p_m\}$ and $V_2 = V(P) \setminus V_1$.
From the observation in the previous paragraph, it follows that each vertex in $V_1$ is connected to every vertex in $V_2$ by an edge, and
for any two vertices in $V_1$, any facet contains at least one of them.
Furthermore, if a facet of $P$ contains exactly $s \geq 1$ elements of $V_1$, then, replacing them with any other $s$ distinct vertices from $V_1$ we obtain
another facet of $P$.
Thus, we may apply Lemma~\ref{lem:note2}, which, by the simplicity of $P$, yields that the linear hulls $L_1$ and $L_2$ of $V_1$ and $V_2$, respectively, are orthogonal.
Clearly, we may assume that the sum of the dimensions of these two subspaces is $d$, as $P$ is $d$-dimensional.
Hence, we have either $\dim L_1 = m-1$ and $\dim L_2 = d+1-m$, or $\dim L_1 = m-2$ and $\dim L_2 = d+2-m$.
Note that in the first case $\conv V_1$ and in the second one $\conv V_2$ is a simplex.

Observe that since $P$ satisfies Property Z, then both $\conv V_1$ and $\conv V_2$ satisfy it in their linear hulls, as otherwise
a slight modification of either $V_1$ or $V_2$ would yield a polytope $P' \in \P_d(d+3)$ with $\vol_d(P) < \vol_d(P')$, contradicting
the definition of Property Z.
Thus, by Corollary~\ref{cor:simplex} and Theorem~\ref{thm:dplus2}, we have that one of $\conv V_1$ and $\conv V_2$ is a regular simplex, and the other one is the convex hull of two regular simplices, contained in orthogonal linear subspaces. Hence, $P$ is the convex hull of three regular simplices, contained in pairwise orthogonal linear subspaces,
and in this case the assertion follows from the argument in the proof of Theorem~\ref{thm:dplus2}.

Observe that since $\bar{P}$ consists of an odd number of points if we do not count multiplicity, if $d$ is odd, then some vertex of $\bar{P}$ has multiplicity strictly greater than one, and thus, in this case the assertion readily follows.
Assume that every vertex of $\bar{P}$ has multiplicity one. Then $d$ is even, and $\bar{P}$ is the vertex set of a regular $d+3$-gon.
We need to show only that $P$ is cyclic.
Since every neighborly $d$-polytope with $n \leq d+3$ vertices is cyclic (cf. \cite{grunbaum}), we show only that $P$ is neighborly.
Indeed, for any $d/2$ vertices of $P$, it is clear that convex hull of the points of $\bar{P}$ corresponding to the remaining $d/2+3$ vertices of $P$
contains $o$ in its interior, since every open half plane, containing $o$ in its boundary, contains either $d/2+1$ or $d/2+2$ points of $\bar{P}$.

The rest of the assertion follows from the volume estimates in the proof of Theorem~\ref{thm:dplus2}.
\end{proof}

\begin{remark}\label{rem:cyclperm}
Let $d=2m$ be even and $P \in \P_d(d+3)$ be a cyclic polytope satisfying Property Z.
Then we need to examine the case that $\bar{P}$ is the vertex set of a regular $(2m+3)$-gon.
Let the vertices of $\bar{P}$ be $\bar{p}_i$, $i=1,2,\ldots,2m+3$ in counterclockwise order.
Applying the method of the proof of Theorem~\ref{thm:dplus3}, one can deduce that for every $i$, we have that $|p_{i-m-1}-p_i| = |p_{i+m+1}-p_i|$.
On the other hand, for any other pair of vertices the conditions of Lemma~\ref{lem:note2} are not satisfied.
\end{remark}

In the light of Theorem~\ref{thm:dplus3}, it seems interesting to find the maximum volume cyclic polytopes in $\P_d(d+3)$, with $d$ even.
With regard to Remark~\ref{rem:cyclperm}, it is not unreasonable to consider the possibility that the answer for this question is a polytope $P = \conv \{p_i : i=1,2,\ldots,d+3 \}$ having a certain cyclic symmetry (if at all it is possible), namely that for any integer $k$, the value of $|p_{i+k}-p_i|$ is independent from $i$.

The following observation can be found both in \cite{KaibelWassmer}, or, as an exercise, in \cite{ziegler}.

\begin{remark}
Let $d \geq 2$ be even, and $n \geq d+3$.
Let
\[
C_d(n) = \sqrt{\frac{2}{d}} \conv \left\{ \left( \cos \frac{i\pi}{n}, \sin \frac{i\pi}{n},  \cos \frac{2i\pi}{n}, \ldots, \cos \frac{di\pi}{2n}, \sin \frac{di\pi}{2n} \right):i=0,1,\ldots,n-1 \right\}.
\]
Then $C(n,d)$ is a cyclic $d$-polytope inscribed in $\Sph^{d-1}$, and $\Sym(C_d(n)) = D_n$.
\end{remark}

In the remaining part of Section~\ref{sec:dplus3}, we show that for $d=4,6$ the only ``symmetric'' representations of a cyclic $d$-polytope with $d$ even and $n=d+3$ are those congruent to $C_d(d+3)$.
Before stating our result, recall that the \emph{L\"owner ellipsoid} of a convex body $K$ is the unique ellipsoid with minimal volume containing $K$
(cf. \cite{thompson}, \cite{ball} or \cite{john}).
The following theorem can be found, for example, in \cite{henk}.

\begin{theorem}[\cite{henk}]\label{thm:Lowner}
Let $K\subset \B^d$ be a compact, convex set. Then $\B^d$ is the L\"owner-ellipsoid of $K$ if, and only if
for some $d\leq n \leq \frac{d(d+3)}{2}$ and $k = 1, \ldots , n$, there are $u_k\in \Sph^{d-1} \cap \bd K $ and $\lambda_k > 0$
such that
\[
0=\sum\limits_{k=1}^{n}\lambda_k u_k \quad \Id =\sum\limits_{k=1}^n\lambda_k u_k\otimes u_k,
\]
where $\Id$ is the identity matrix, and for $u, v \in \R^d$, $u\otimes v$ denotes the $d\times d$ matrix $u v^T$.
\end{theorem}

We need the following.

\begin{lemma}\label{lem:Lowner}
Let $S=\{ u_1,\ldots,u_n\} \subset \Sph^{d-1}$ and $\lambda_1,\ldots,\lambda_n > 0$.
Let $G$ be the Gram matrix of $S$, $\Lambda_v = (\lambda_1,\ldots,\lambda_n)^T$ and $\Lambda_m$ be the diagonal $n \times n$ matrix 
with $\lambda_i$ as its $i$th entry.
Then the conditions of Theorem~\ref{thm:Lowner} are satisfied for the $u_i$s and $\lambda_i$s if, and only if
$S$ is $d$-dimensional, and
\begin{equation}
G \Lambda_v = o, \quad \mbox{ and } \quad G \Lambda_m G = G.
\end{equation}
\end{lemma}

\begin{proof}
Clearly, if $S$ is not $d$-dimensional, then the second condition of Theorem~\ref{thm:Lowner} is not satisfied.
On the other hand, if $S$ is $d$-dimensional, then $G \Lambda_v = o$ is equivalent to the condition that
$0=\sum_{k=1}^n \lambda_k \langle u_k,u_i \rangle$ for every $i$, which is in turn equivalent to $o = \sum_{k=1}^n \lambda_k u_k$.

Consider the equality $\Id =\sum\limits_{k=1}^n\lambda_k u_k\otimes u_k$.
From this, we obtain for every $i$ and $j$ that
\[
\langle u_i,u_j \rangle = \sum_{k=1}^n \lambda_k \langle u_i, (u_k\otimes u_k)u_j \rangle = \sum_{k=1}^n \lambda_k \langle u_i,u_k \rangle \langle u_k,u_j \rangle,
\]
which implies that $G = G \Lambda_m G$. The opposite direction can be shown similarly.
\end{proof}

\begin{theorem}\label{thm:symmetric}
Let $P \in \C_d(d+3)$ be a cyclic polytope, where $d=4$ or $d=6$, and let $V(P)=\{p_i:i=1,2,\ldots,d+3 \}$.
If, for every value of $k$, $|p_{i+k}-p_i|$ is independent of the value of $i$, then $P$ is congruent to $C_d(d+3)$.
\end{theorem}

\begin{proof}
First, we examine the case that $d=4$.
Observe that the values of the edge lengths of $P$ determine the facets up to congruence.
Thus, under our conditions, any combinatorial symmetry of the vertices can be realized geometrically, which yields that $D_7 \leq \Sym(P)$.
Since it implies that $o$ is the only fixed point of the elements of $\Sym(P)$, we have that the L\"owner ellipsoid of $P$ is $\B^4$, and we can apply
Theorem~\ref{thm:Lowner} and Lemma~\ref{lem:Lowner}.

Clearly, the points $u_i$ in the theorem and the lemma are vertices of $P$.
On the other hand, if we add extra points to them, with zero coefficients, both the conditions in the theorem and in the lemma remain true.
Thus, we may apply Lemma~\ref{lem:Lowner} with $n=7$, the points of $V(P)$ as $u_i$s, while permitting some of the $\lambda_i$s to be zero.

Let $X \subset \R^7$ be the set of points $(\lambda_1,\ldots,\lambda_7)$ such that $V(P)$, with these coefficients, satisfies the conditions of Lemma~\ref{lem:Lowner}. It is an elementary computation to show that $X$ is convex.
Thus, it follows from $D_7 \leq \Sym(P)$ that if $(\lambda_1,\ldots,\lambda_7)$ satisfies the conditions, then the same holds for
$(\lambda,\ldots,\lambda)$, where $\lambda = \frac{1}{7} \sum_{i=1}^7  \lambda_i > 0$.
 
Let the Gram matrix of $V(P)$ be
\[
G = \left[
\begin{array}{ccccccc}
1 & a & b & c & c & b & a\\
a & 1 & a & b & c & c & b\\
\vdots & \vdots & & & & & \vdots\\
a & b & c & c & b & a & 1
\end{array}
\right]
\] 
Then the equation $G \Lambda_v = 0$, with $\Lambda_v = (\lambda,\ldots,\lambda)^T$  yields that $1+2a+2b+2c=0$.
Under this condition, we obtain the following solutions for the equation $G \Lambda_m G = G$:
\begin{itemize}
\item $a=b=c=-\frac{1}{6}$;
\item the values of $a,b,c$ are the three different real roots of the equation $8x^3+4x^2-4x-1=0$;
\item the values of $a,b,c$ are the three different real roots of the equation $8x^3+8x^2-2x-1=0$.
\end{itemize}
It is a matter of computation to check that the first solution determines a regular simplex in $\R^6$, the second one determines a regular heptagon in $\R^2$
(or equivalently, $C_2(7)$), and the third one determines $C_4(7)$.

To prove the assertion for $d=6$, we may apply the same argument. In this case we obtain seven solutions for $G$.
On contains coinciding points, and another one cannot be realized in any Euclidean space.
The remaining solutions are:
\begin{enumerate}
\item a regular simplex in $\R^8$,
\item polytopes congruent to $C_2(9)$, $C_4(9)$ and $C_6(9)$,
\item the convex hull of three regular triangles, with $o$ as their center, and contained in three mutually orthogonal linear subspaces of $\R^6$.
\end{enumerate}
Thus, the only solutions satisfying all the conditions for $P$ are the ones congruent to $C_6(9)$.
\end{proof}

\section{Remarks and open problems}\label{sec:remarks}

We start with a question that is a direct consequence of Theorem~\ref{thm:dplus3}.

\begin{problem}\label{prob:cyclic}
Prove or disprove that in $\P_d(d+3)$, the cyclic polytopes with maximal volume are the congruent copies of $C_d(d+3)$.
In particular, is it true for $C_4(7)$? Is it true that any cyclic polytope in $\P_d(d+3)$ satisfying Property Z is congruent to $C_d(d+3)$?
\end{problem}

\begin{remark}
A straightforward computation shows that $C_4(7)$ satisfies Property Z.
\end{remark}

We note that in $\P_d(d+2)$, polytopes with maximal volume are cyclic, whereas in $\P_d(d+3)$, where $d$ is odd, they are not.
This leads to Problem~\ref{prob:whichone}.

\begin{problem}\label{prob:whichone}
Is it true that if $P \in \P_d(d+3)$, where $d$ is even, has volume $v_d(d+3)$, then $P$ is not cyclic?
\end{problem}

\begin{remark}\label{rem:notcyclic}
Let $P_4 \in \P_4(7)$ be the convex hull of a regular triangle and two diameters of $\Sph^3$, in mutually orthogonal linear subspaces.
Furthermore, let $P_6 \in \P_6(9)$ be the convex hull of three regular triangles, in mutually orthogonal linear subspaces.
One can check that
\[
\vol_4(P_4)= \frac{\sqrt{3}}{4} = 0.43301\ldots > \vol_4(C_4(7)) = \frac{49}{192} \left( \cos \frac{\pi}{7} + \cos \frac{2\pi}{7} \right) = 0.38905\ldots.
\]
In addition,
\[
\vol_6(C_6(9))= \frac{7}{576} \sin \frac{\pi}{9} - \frac{7}{2880} \sin \frac{4\pi}{9} + \frac{7}{1152} \sin \frac{2\pi}{9} = 0.01697\ldots
\]
and
\[
\vol_6(P_6) = \frac{9\sqrt{3}}{640} = 0.02435\ldots > \vol_6(C_6(9)).
\]
This suggests that the answer for Problem~\ref{prob:whichone} is yes.
\end{remark}

\begin{remark}
Using the idea of the proof of Theorem~\ref{thm:symmetric}, for any small value of $n$, we may identify the polytopes
having $D_n$ as a subgroup of their symmetry groups. Nevertheless, we were unable to apply this method for general $n$, due to computational complexity.
We carried out the computations for $5 \leq n \leq 9$, and obtained the following polytopes, up to homothety:
\begin{itemize}
\item regular $(n-1)$-dimensional simplex in $\R^{n-1}$ for every $n$,
\item regular $n$-gon in $\R^2$ for every $n$,
\item $C_4(n)$ with $n=6,7,8,9$ and $C_6(n)$ with $n=8,9$,
\item regular cross-polytope in $\R^3$ and $\R^4$,
\item the polytope $P_6$ in $\R^6$, defined in Remark~\ref{rem:notcyclic},
\item the $3$-polytope $P$ with
\[
V(P)=\left\{ \left(1,0,0 \right),\left(-\frac{2}{3},-\frac{2}{3},\frac{1}{3} \right),  \left(0,1,0 \right), \left(\frac{1}{3},-\frac{2}{3},-\frac{2}{3} \right), \left(0,0,1 \right), \left(-\frac{2}{3},\frac{1}{3},-\frac{2}{3} \right)\right\} .
\]
\end{itemize}
\end{remark}

We note that, for $d$ odd, the symmetry group of a cyclic $d$-polytope with $n \geq d+3$ vertices is $\Z_2 \times \Z_2$ (cf. \cite{KaibelWassmer}).
Thus, the only cyclic polytopes in the above list are simplices and those homothetic to $C_d(n)$ for some values of $n$ and $d$. 
This leads to the following question.

\begin{problem}
Is it true that if, for some $n \geq d+3 \geq 5$, a cyclic polytope $P \in \P_d(n)$ satisfies $\Sym(P) = D_n$, then $P$ is congruent to $C_d(n)$?
\end{problem}

\section{Acknowledgments}
The support of the J\'anos Bolyai Research Scholarship of the Hungarian Academy of Sciences is gratefully acknowledged.
The authors are indebted to T. Bisztriczky for his help in understanding the combinatorial properties of neighborly polytopes.

\end{document}